\documentclass[11pt,letter,margin=1in]{article}

\usepackage{fullpage}
\usepackage{amssymb,amsmath,amsthm}
\usepackage{epsfig}
\usepackage{url}
\usepackage{mathtools}
\usepackage{breqn}

\usepackage{color}

\makeatletter
\let\@@citation@@=\citation
\renewcommand{\citation}[1]{\@@citation@@{#1}%
\@for\@tempa:=#1\do{\@ifundefined{cit@\@tempa}%
  {\global\@namedef{cit@\@tempa}{}}{}}%
}
\makeatother

\usepackage{hyperref}

\makeatletter
\def\@lbibitem[#1]#2#3\par{%
  \@ifundefined{cit@#2}{}{\@skiphyperreftrue
  \H@item[%
    \ifx\Hy@raisedlink\@empty
      \hyper@anchorstart{cite.#2\@extra@b@citeb}%
        \@BIBLABEL{#1}%
      \hyper@anchorend
    \else
      \Hy@raisedlink{%
        \hyper@anchorstart{cite.#2\@extra@b@citeb}\hyper@anchorend
      }%
      \@BIBLABEL{#1}%
    \fi
    \hfill
  ]%
  \@skiphyperreffalse}%
  \if@filesw
    \begingroup
      \let\protect\noexpand
      \immediate\write\@auxout{%
        \string\bibcite{#2}{#1}%
      }%
    \endgroup
  \fi
  \ignorespaces
  \@ifundefined{cit@#2}{}{#3}}

\def\@bibitem#1#2\par{%
  \@ifundefined{cit@#1}{}{\@skiphyperreftrue\H@item\@skiphyperreffalse
  \Hy@raisedlink{%
    \hyper@anchorstart{cite.#1\@extra@b@citeb}\relax\hyper@anchorend
    }}%
  \if@filesw
    \begingroup
      \let\protect\noexpand
      \immediate\write\@auxout{%
        \string\bibcite{#1}{\the\value{\@listctr}}%
      }%
    \endgroup
  \fi
  \ignorespaces
  \@ifundefined{cit@#1}{}{#2}}
\makeatother

\newtheorem{thm}{Theorem}
\newtheorem{cor}[thm]{Corollary}
\newtheorem{lem}[thm]{Lemma}
\newtheorem{prop}[thm]{Proposition}
\newtheorem{claim}[thm]{Claim}

\theoremstyle{definition}

\newtheorem{defi}[thm]{Definition}
\newtheorem{example}[thm]{Example}
\newtheorem*{example*}{Example}

\mathcode`l="8000
\begingroup
\makeatletter
\lccode`\~=`\l
\DeclareMathSymbol{\lsb@l}{\mathalpha}{letters}{`l}
\lowercase{\gdef~{\ifnum\the\mathgroup=\m@ne \ell \else \lsb@l \fi}}%
\endgroup

\begin{document}

\title{At most $3.55^n$ stable matchings}
\author{Cory Palmer\thanks{Department of Mathematical Sciences, University of Montana, Missoula, Montana 59812, USA. E-mail: \texttt{cory.palmer@umontana.edu.}} \and D\"om\"ot\"or P\'alv\"olgyi\thanks{MTA-ELTE Lend\"ulet Combinatorial Geometry Research Group, Institute of Mathematics, E\"otv\"os Lor\'and University, Budapest, Hungary. E-mail: \texttt{domotorp@gmail.com}.}
}
\maketitle

\begin{abstract}
We improve the upper bound for the maximum possible number of stable matchings among $n$ jobs and $n$ applicants from $131072^n+O(1)$ to $3.55^n+O(1)$.
To establish this bound, we state a novel formulation of a certain entropy bound that is easy to apply and may be of independent interest in counting other combinatorial objects.
\end{abstract}



\pagenumbering{arabic}

\section{Introduction}

A {\it stable matching instance} is a complete bipartite graph $G$ with an $n$-vertex class of {\it jobs}\footnote{In this paper we have decided to use the terms of {\it jobs} and {\it applicants} in place of the more traditional {\it men} and {\it women}; a similar approach was used in \cite{DM}.} and an $n$-vertex class of {\it applicants}$^1$ where each vertex has a strict order of {\it preferences} on the vertices in the opposite class.
Given a perfect matching 
in $G$, a pair $u,v$ is {\it unstable} if $u$ and $v$ both prefer each other to their currently matched pair. A perfect matching that contains no unstable pairs is called a {\it stable matching}.
Gale and Shapley \cite{GS} showed that a stable matching exists in every stable matching instance, and gave an efficient algorithm to find one; see also \cite[Chapter 1]{KT}.
Since their introduction in 1962 by Gale and Shapley, stable matchings have found a wide variety of practical applications including college admissions \cite{GS}, hospital residences \cite{MW,R1,R2}, kidney exchanges \cite{B1,B2}, and even pure math proofs \cite{Galvin}.
Shapley and Roth were awarded the 2012 Nobel Memorial Prize in Economic Sciences for their work on stable matchings.

A long-standing problem, first posed by Knuth in 1976 \cite{Knuth} (see also Manlove \cite[Section 2.2.2]{M}) is to determine the maximum possible number $SM(n)$ of stable matchings in a stable matching instance with $n$ jobs and $n$ applicants. This problem is interesting in its own right and also has implications for the run time of certain algorithms \cite{Kavitha}. 

A very easy lower bound for $SM(n)$ is $2^{n/2}$, while the best-known lower bound $\Omega(2.28^n)$ is due to Irving and Leather \cite{IL}. This bound holds when $n$ is a power of $2$, but for all other $n$, Thurber \cite{Th} gave a slightly weaker bound of $2.28^n/c^{\log n}$.

The trivial upper bound on the number of stable matchings is $n!$. This was improved to $O(n!/c^n)$ by Strathoupolos \cite{St} and $\frac{3}{4}n!$ by Drgas-Burchardt and \'Switalski \cite{DBS}.
Despite considerable attention \cite{BCK,DBS,Hw,St,Th}, the first simply exponential upper bound was obtained only recently by Karlin, Gharan and Weber \cite{KGW}. In that paper, the authors remark that their argument gives a bound of $c^n=2^{17n}$ for $n$ large enough, but that their method will not give a $c$ close to the best-known lower bound. Their approach is a probabilistic argument, similar to one of the early proofs of the Crossing Lemma \cite{ACNS}.
A recent survey of Cechl\'arov\'a,
Cseh and Manlove~\cite{CCM} highlights the importance of further improvements to the upper bound. 

Our main result is to improve the upper bound for the maximum possible number of stable matchings; the base of our exponent gets close to $2.28$, the base in the best lower bound.

\begin{thm}\label{thm:main}
The number of stable matchings $SM(n)$ in an instance with $n$ jobs and $n$ applicants satisfies
\[
SM(n) \leq 3.55^n +O(1).
\]
\end{thm}

Apart from its theoretical importance, this also improves the known upper bound for the running time of certain algorithms. For example, the upper bound of the running time of the best-known algorithm for the popular roommates problem with strict preferences by Kavitha \cite{Kavitha} is improved from roughly $400000^n$ to $10.65^n$.

As mentioned in \cite{KGW}, Theorem \ref{thm:main} also gives an upper bound for the so-called \emph{stable roommates problem}, defined similarly to the stable matching problem, but without requiring the graph $G$ to be bipartite.
This follows by a simple reduction to the stable matching problem where we take two copies of each vertex \cite{DM}.

\begin{cor}
The number of stable matchings among $n$ roommates is at most $3.55^n +O(1)$.
\end{cor}

Our starting point to prove Theorem \ref{thm:main} is similar to the starting point of \cite{KGW}.
We will observe some properties of the rotation poset (to be defined in Section~\ref{sec:simple}), and then use them to give an upper bound for $SM(n)$ with a variant of the entropy method. 
Our Lemma~\ref{lem:main} may be of independent interest, as it provides a simple-to-apply method for counting combinatorial objects compared to other techniques, like those that employ entropy or the container method \cite{container}.
\\

 \textbf{Notation.} 
We write the expected value of a function $X$ of a random variable $r$ as $E_r[X(r)]$ when we want to emphasize that $r$ is the random variable---this is to avoid confusion when $X$ also depends on other (fixed) parameters.
All logarithms in this paper are natural and contain no artificial ingredients or added color. \\

\textbf{Outline of the paper.}  
In Section \ref{sec:method} we present our main result for counting objects, Lemma \ref{lem:main}, and give an example of a simple application, a short proof of Bregman's theorem.
In Section \ref{sec:simple} we first outline why earlier results \cite{GI,IL,KGW} imply that the number of stable matchings is bounded by the number of downsets of a so-called \emph{tangled grid} poset.
We then use our main lemma to prove that the number of downsets in a tangled grid is at most $11.11^n$. This gives a bound on $SM(n)$ and is a combinatorial problem of independent interest.
In Section \ref{sec:tg2} we refine these methods to prove Theorem~\ref{thm:main}.

\section{Main Lemma}\label{sec:method}

In this section we present a method for bounding the cardinality of a family $S$.
To this end, we will encode $S$ as a family of $n$-tuples from $A_1 \times A_2 \times \cdots \times A_n$ for some carefully selected sets $A_i$. We use the term {\it component} both to refer to the sets $A_i$ and to the elements $s_i$ for $s = (s_1,s_2,\dots, s_n) \in S$. It will be convenient to think of each of the components $s_i$ of $s$ as carrying some information about $s$ such that $s$ is completely determined by the values $s_1,s_2,\dots, s_n$.

If $S\subset A_1 \times \cdots \times A_n$, then a trivial upper bound on $|S|$ is $|A_1|\cdots|A_n|$.
However, if the structure of $S$ has sufficient symmetry, then we can get better bounds by revealing the components one-by-one as follows.



Fix a permutation $\pi$ of $\{1,\ldots,n\}$ and an $n$-tuple
$s=(s_1,s_2,$ $\dots, s_n) \in S$.
Imagine that the components of $s$ are revealed according to $\pi$ in the order $s_{\pi(1)},s_{\pi(2)},\ldots,s_{\pi(n)}$.
Let $X_i(s,\pi)$ count the number of possible $i$th components of members of $S$ whose first $\pi^{-1}(i)-1$ components under $\pi$ (i.e., those components that are revealed before $s_i$ is revealed) are identical to those of $s$ under $\pi$. 
That is,
\[
X_i(s,\pi)=|\{x_i \,:\, (x_1,\dots, x_n)\in S,  x_j=s_j \textit{ for } j \textit{ satisfying } \pi^{-1}(j)<\pi^{-1}(i)\}|.
\]

This immediately improves the trivial bound of $|S| \leq |A_1|\cdots |A_n|$ to 
\begin{equation}\label{easy-upper}
|S|\le \prod_{i=1}^n \max_{s \in S} X_i(s,\pi)
\end{equation}
for any permutation $\pi$. We give an example that illustrates how to compute $X_i(s,\pi)$.

\begin{example}\label{example1}
	Put
	\[
	S=\left\{(i,i,0) \,:\, 1 \leq i \leq N\right\} \cup \left\{(i,0,i) \,:\, 1 \leq i \leq N\right\} \cup \left\{(0,i,i) \,:\, 1 \leq i \leq N\right\}.
	\] 
	Thus, $S \subset \{0,1,\ldots,N\}^3$ and $|S|=3N$.
    Suppose that $s=(s_1,s_2,s_3)=(i,i,0)$ for some $i\in \{1,\ldots,N\}$ and consider all 6 possible permutations $\pi$ of $\{1,2,3\}$.
    If $\pi(1)=1$, i.e., coordinate 1 is revealed first, then $X_1(s,\pi)=N+1$, as $s_1$ can take any value from 0 to $N$.
    If $\pi(3)=1$, i.e., coordinate 1 is revealed last, then $X_1(s,\pi)=1$, as $s_2=i$ and $s_3=0$ determine that $s_1=i$.
    Otherwise, $\pi(2)=1$, so $\pi=(2,1,3)$ or $\pi=(3,1,2)$.    
    If $\pi=(2,1,3)$, from $s_2=i$ we learn that $s_1$ is either 0 or $i$, so there are only two options left for $s_1$.
    Finally, if $\pi=(3,1,2)$, from $s_3=0$ we only learn that $s_1\ne 0$, so there are still $N$ options for $s_1$.
    We can repeat a similar analysis for $s=(i,0,i)$ and $s=(0,i,i)$ to calculate all values of $X_1(s,\pi)$; see Table~\ref{table:example}. Since $S$ is symmetric under a permutation of the coordinates, the possible values of $X_2(s,\pi)$ and $X_3(s,\pi)$ would be given by a table that can be obtained from Table~\ref{table:example} by a permutation of rows and columns.    
    
\begin{table}[h]
\begin{center}
\begin{tabular}{|c|c|c|c|c|c|c|}
\hline
$X_1(s,\pi)$ & 123 & 132 & 213 & 231 & 312 & 321 \\ \hline
$(i,i,0)$    & $N+1$ & $N+1$ & 2   & 1   & $N$ & 1   \\ \hline
$(i,0,i)$    & $N+1$ & $N+1$ & $N$ & 1   & 2   & 1   \\ \hline
$(0,i,i)$    & $N+1$ & $N+1$ & 2   & 1   & 2   & 1   \\ \hline
\end{tabular}
\caption{Summary of values of $X_1(s,\pi)$. Rows represent the elements $s\in S$, while columns represent the permutations $\pi$ of $\{1,2,3\}$.}\label{table:example}
\end{center}
\end{table}    
\vspace{-5mm}

Observe that the trivial upper bound gives $|S|\le (N+1)^3$, while the slightly better bound in (\ref{easy-upper}) gives $|S|\le (N+1)\cdot N\cdot 1 = O(N^2)$ for any permutation $\pi$.	
	
Lemma~\ref{lem:main} (to be stated below) further improves this bound on $|S|$.
The first part of Lemma~\ref{lem:main} gives $\log |S|\le \log(N+1)+\frac13\log N+\frac23\log 2+\log 1$ (for any choice of $\pi$), which implies $|S|\le 2^{\frac23} (N+1)N^{\frac13} = O(N^{\frac 43})$.
The second part of Lemma~\ref{lem:main}, when we average over only $\pi$, chosen uniformly at random, gives the weaker
$\log |S|\le 3 \cdot \frac16 (2\log(N+1)+\log N +\log 2+2\log 1)$, which implies $|S|\le (2(N+1)^2N)^{\frac 12} = O(N^{\frac 32})$.	
The loss compared to the previous bound is due to taking the maximum over all $s\in S$ instead of the expectation, but we have stated this result in this form because in some applications (like in the proof of Theorem \ref{thm:main}) we have no control over $s$.
Finally, Corollary~\ref{cor:main} (to be stated below), which is easier to apply than Lemma~\ref{lem:main}, gives the much weaker bound $|S|\le (\frac{3N+6}6)^3$.
\end{example}	

Now we state our main lemma.
	
\begin{lem}\label{lem:main}
Let $S \subset A_1 \times \cdots \times A_n$ and $\pi$ be any fixed permutation of $\{1,\ldots,n\}$. 
Choose $s \in S$ uniformly at random.
Then
    \[
    \log |S|\le E_s\left[ \sum_{i=1}^n \log X_i(s,\pi)\right].
    \]
    By averaging, if $\pi$ is chosen at random from the permutations of $\{1,\ldots,n\}$ according to an arbitrary probability distribution,
	\[
	\log |S|\le E_{(s,\pi)}\left[ \sum_{i=1}^n \log X_i(s,\pi)\right]\le
	\sum_{i=1}^n \max_{s\in S} E_\pi\left[ \log X_i(s,\pi)\right].
	\]
\end{lem}

Applying Jensen's inequality gives the following weaker but simpler form of Lemma~\ref{lem:main}.

\begin{cor}\label{cor:main}
Let $S \subset A_1 \times \cdots \times A_n$. Then for any probability distribution of the permutations
    \[
    |S|\le 
	\prod_{i=1}^n \max_{s\in S} E_\pi[X_i(s,\pi)].
	\]
\end{cor}

Although we could not find Lemma \ref{lem:main} explicitly stated anywhere, very similar bounds are frequently used in entropy proofs.
Indeed, the first line of Lemma \ref{lem:main} follows from Shannon's famous noiseless coding theorem \cite{Shannon}.
In that language Lemma \ref{lem:main} says that if the letters of the alphabet $S$ are evenly distributed in a random text, then any encoding uses at least $\log |S|$ bits per character on average.
Lemma \ref{lem:main} can be also easily proved using the basic properties of entropy; we sketch such a proof in Appendix \ref{app:main}.

Taking averages over all permutations is also a widely used tool in proofs. Arguably, the core idea of this goes back to Katona's circle method \cite{Katona}; see also the permutation method described in \cite[Chapter 1]{GP}.
Therefore, most of the arguments presented in this paper could be repeated using entropy and averaging, but we feel that Lemma \ref{lem:main} can be a useful tool for proving upper bounds because it can be applied without introducing entropy.

Let us mention one form of Lemma \ref{lem:main}, not explicitly stated in the book by Alon-Spencer \cite[Chapter 2, The Probabilistic Lens]{AS}.
They defined for a non-negative random variable $X$ taking finitely many possible values, all with non-zero probability, its \emph{geometric mean} as $G[X]=\prod_x x^{Pr[X=x]}$. As noted there, $G[X]=e^{E[\log X]}\le E[X]$. 
In this language, Lemma \ref{lem:main} states
$|S|\le \prod_{i=1}^n G[X_i(s,\pi)]$ where $s$ is a uniform random variable and $\pi$ is an arbitrary permutation. 


Lemma~\ref{lem:main} shows that in order to
 get an upper bound on $|S|$, it is enough to find components $A_i$ so that $E_\pi[\log X_i(s,\pi)]$ is small for every $s \in S$.
As a simple application, we prove Bregman's theorem \cite{BR}. The argument follows Alon-Spencer \cite{AS} and Radhakrishnan \cite{R}, the only difference is that entropy is replaced by a simple application of Lemma \ref{lem:main}. 

\begin{thm}[Bregman \cite{BR}]
	The number of perfect matchings in a bipartite graph $G=(U,V;E)$ is at most $\prod_i (d_i!)^{1/d_i}$ where $d_i$ denotes the degree of vertex $v_i \in V$.
\end{thm}
\begin{proof}
Let $S$ be the set of all perfect matchings in $G$ and
    let $A_i$ be the set of edges incident to $v_i \in V$.
    Then $S \subset \bigtimes_i A_i$.
	Fix a perfect matching $s \in S$ and order the vertices $v_i \in V$ according to a random permutation $\pi$.
	We first reveal the edge of $s$ incident to $v_{\pi(1)}$, then the edge of $s$ incident to $v_{\pi(2)}$ and so on.
	Recall that $X_i(s,\pi)$ is the number of options remaining for the edge of $s$ incident to $v_i$
	right before it is revealed.
	Let $N_i(s,\pi)$ be the number the neighbors of $v_i$ (necessarily in $U$) that have not yet had their edge from $s$ revealed. Then we have $X_i(s,\pi)\le N_i(s,\pi)$, where the inequality can be strict, as some of these edges may not be extendable to a perfect matching.
	
 	Since the edges of $s$ were ordered randomly, $Pr[N_i(s,\pi)=k]= 1/{d_i}$ for every $1\le k\le d_i$.
	Therefore, by Lemma \ref{lem:main} we have that the number of perfect matchings $|S|$ satisfies
	\begin{align*}
	\log |S|
	& \le \sum_{i} \max_{s \in S} E_\pi[ \log X_i(s,\pi)]\le \sum_{i} \max_{s \in S} E_\pi[ \log N_i(s,\pi)] = \sum_{i} \sum_{k=1}^{d_i} \frac 1{d_i}\log k=\sum_{i} \log\left((d_i!)^{1/d_i}\right).\phantom\qedhere 
	\end{align*}
\end{proof}

\section{General method and a simple bound for $SM(n)$}\label{sec:simple}

In this section we give our first bounds on $SM(n)$ by counting downsets in the so-called tangled grid poset.

\subsection{From stable matchings to downsets in the tangled grid poset}\label{sec:down-to-tg}

In this subsection we detail a connection between stable matchings and downsets of a tangled grid.
We begin by discussing how stable matchings can be mapped to the downsets in the rotation poset, a result established by Irving and Leather \cite{IL}. (It is also possible to show a correspondence in the other direction, see \cite{Blair,GILS}.)
Following \cite{KGW}, we define a rotation in a stable matching.



\begin{defi}[Rotation]
    Let $k \geq 2$. A \emph{rotation} $\rho$ is an ordered list of edges 
    \[
    \rho = (u_0v_0, u_1v_1,\dots, u_{k-1}v_{k-1})
    \]
    in a stable matching $\mu$
     with the property that for $i=0,1,2,\dots, k-1$ applicant $v_{i+1}$ (where the indices are modulo $k$) is the highest ranked applicant for $u_i$'s list of preferences satisfying:
    \begin{enumerate}
        \item job $u_i$ prefers $v_{i}$ to $v_{i+1}$, and
        \item applicant $v_{i+1}$ prefers $u_{i}$ to $u_{i+1}$.
    \end{enumerate}
    In this case we say that $\rho$ is \emph{exposed} in the stable matching $\mu$.
\end{defi}



We say that a rotation $\rho = (u_0v_0, u_1v_1,\dots, u_{k-1}v_{k-1})$ is \emph{eliminated} if we replace the edges of the stable matching $\mu$ appearing in $\rho$ with
$u_{0}v_{1},u_{1}v_{2},\ldots, u_{k-1}v_{0}$. It is easy to see that the perfect matching resulting from the elimination of a rotation $\rho$ from a stable matching is itself stable (see Lemma~4.5 in \cite{IL}).
It was shown in \cite{GI} that any stable matching can be obtained from what we will call a \emph{job-optimal} stable matching $\mu_0$ through a sequence of eliminations of exposed rotations. This leads to a natural poset on the family $\mathcal{R}$ of rotations exposed in any stable matching of a stable matching instance. Indeed, for rotations $\rho,\rho' \in \mathcal{R}$ we say that $\rho \prec \rho'$ if and only if for every stable matching $\mu$ in which $\rho'$ is exposed we must eliminate $\rho$ to arrive at $\mu$ from $\mu_0$. We call the poset $\mathcal{R}$ {\it rotation poset}.


Recall that a subset of a poset is a {\it downset} (or {\it downward closed set}) if for every element of the downset, all smaller elements of the poset are in the downset.

\begin{thm}[Irving \& Leather \cite{IL}, Theorem~4.1]\label{bijection-thm}
Given a stable matching instance, there is a bijection between the family of stable matchings and the downsets of the rotation poset.
\end{thm}

As in \cite{KGW}, this correspondence is central to our argument---instead of investigating stable matchings directly, we may estimate the number of downsets in the rotation poset.
We will need the following immediate consequence of an observation explicitly stated in \cite{KGW}, and implicit in \cite{GI}.
We say that a rotation $\rho$ {\it involves} vertex $v$ if $v$ appears among the vertices in the ordered list of edges in $\rho$.

Recall that a {\it chain} in a poset is a sequence of comparable elements $a_1 \prec a_2 \prec \cdots \prec a_k$. We will use the term {\it top} and {\it bottom} of a chain to refer to elements $a_1$ and $a_k$, respectively.

\begin{prop}[Karlin, Gharan \& Weber \cite{KGW}, Claim~4.6]\label{kgw-claim}
    The rotations involving a fixed vertex $v$ form a chain in the rotation poset.
\end{prop}

So each vertex of a stable matching instance defines a chain in the rotation poset. 
 For a job $u$, let $M(u)$ be the chain of rotations involving $u$. The family of chains $\{M(u) \, : \, u \textrm { is a job} \}$ are the {\it m-chains} in the rotation poset. Similarly, for an applicant $v$, let $W(v)$ be the chain of rotations involving $v$ and the family of chains $\{W(v) \, : \, v \textrm { is an applicant} \}$ are the {\it w-chains} in the rotation poset.

\begin{prop}[Irving \& Leather \cite{IL}, Lemma~4.7]\label{one-pair}
A edge $uv$ where $u$ is a job and $v$ is an applicant appears in at most one rotation.
\end{prop}

Recall that a \emph{chain decomposition} of a poset is a partition of its elements into pairwise disjoint chains.

\begin{defi}[Tangled grid]
    A poset $P$ is an \emph{$n\times n$ tangled grid} if $P$ has two  $n$-member chain decompositions called {\it m-chains} and {\it w-chains}, such that each m-chain and each w-chain intersect each other in exactly one element.
\end{defi}

\begin{prop}\label{subposet}
    Every rotation poset of $n$ jobs and $n$ applicants is a subposet of an $n\times n$ tangled grid.
\end{prop}

\begin{proof}
    We will modify the collection of m-chains and w-chains in a rotation poset $\mathcal{R}$ to arrive at m-chains and w-chains satisfying the definition of an $n \times n$ tangled grid.
    
    For each rotation $\rho = (u_0v_0, u_1v_1,\dots, u_{k-1}v_{k-1}) \in \mathcal{R}$, remove $\rho$ from each m-chain and w-chain except for the m-chain $M(u_0)$ and w-chain $W(v_0)$. By Proposition~\ref{one-pair}, $u_0v_0$ never appears in another rotation, so $M(u_0)$ and $W(v_0)$ intersect only in $\rho$.
    Now every element of $\mathcal{R}$ is contained in exactly one m-chain and in exactly one w-chain.
    Note that some of the chains might have become temporarily empty. At this point, the m-chains and w-chains both form chain decompositions of $\mathcal{R}$.
    
    Now it remains to show that the resulting poset is a subposet of a tangled grid. It is possible that we have an m-chain and a w-chain that do not intersect. In this case, add a new element to the poset that is included in these two chains, and is larger than every existing element of the poset. Continuing this procedure until every m-chain w-chain pair has an intersection results in a tangled grid.
\end{proof}

In an $n \times n$ tangled grid $P$, the m-chains are pairwise-disjoint, and the w-chains are also pairwise-disjoint.
Moreover, since each m-chain and w-chain intersect in a single element, each m-chain and w-chain is of length exactly $n$.
These together imply that $P$ has exactly $n^2$ elements.

Let $TG(n)$ be the maximum number of downsets in an $n \times n$ tangled grid and let $SM(n)$ be the maximum number of stable matchings in a stable matching instance with $n$ jobs and $n$ applicants. Proposition~\ref{subposet} and Theorem~\ref{bijection-thm} imply that $SM(n) \leq TG(n)$, so we may turn our attention to counting downsets in a tangled grid. This will be our approach in the next subsection.


\subsection{Simple bound for $TG(n)$}\label{sec:simpletg}

Let $P$ be an $n \times n$ tangled grid. For the sake of simpler notation, let us construct a new poset $P^*$ by adding two new elements $a_0$ and $b_0$ to $P$ that are less than all other elements in $P$.

Add $a_0$ to the bottom of each m-chain in $P$ and let $M_1,M_2,\dots, M_n$ be the resulting chains in $P^*$. Let $W_{n+1},W_{n+2},\dots, W_{2n}$ be the chains resulting from adding $b_0$ to the bottom of the w-chains in $P$. With a slight abuse of notation, we will continue to use the term m-chain and w-chain to describe these chains in $P^*$.
Observe that every pair of chains intersects in exactly one element. 

 Every downset in $P$ corresponds to a downset in $P^*$ that contains $a_0$ and $b_0$. Let $\mathcal{D}$ be the family of downsets in $P^*$ that contain both $a_0$ and $b_0$.
For a downset $D \in \mathcal{D}$, let $\mathrm{top}_D(M_i)$ and $\mathrm{top}_D(W_j)$ be the maximum element of $D$ on chain $M_i$ and $W_j$, respectively. Define 
\[
s_D = (\mathrm{top}_D(M_1),\dots, \mathrm{top}_D(M_n),\mathrm{top}_D(W_{n+1}),\dots, \mathrm{top}_D(W_{2n})).
\]
As $D$ is a downset, the intersection of $D$ and a chain $C$ is characterized by the maximum element $\mathrm{top}_D(C)$. Therefore, $s_D$ uniquely determines $D$. (In fact, the first $n$ terms or last $n$ terms of $s_D$ are enough to determine $D$, but our argument relies on the full encoding above). Thus, it is enough to find the cardinality of $S = \{s_D \, : \, D \in \mathcal{D}\}$. Clearly, $S \subset M_1 \times \cdots \times M_n \times W_{n+1} \times \cdots \times W_{2n}$.

Reveal the components $M_1,\dots,M_n,W_{n+1},\dots, W_{2n}$ of $S$ in order  according to a random permutation $\pi$ of $\{1,\ldots,2n\}$.
Recall that $X_i(s,\pi)$ is the number of options remaining for the $i$th component of $s$ right before it is revealed. Suppose that $i\leq n$, i.e., we have an m-chain (the case when $i>n$ is similar). So we are counting the number of options for $t=\mathrm{top}_D(M_i)$ in the m-chain $M=M_i$.
 
Denote the number of w-chains preceding $M$ according to $\pi$ by $L(i,\pi)=|\{j\in [n+1,2n] \,:\, \pi^{-1}(j)<\pi^{-1}(i)\}|$. 
 Define $X_i(s,\pi\mid L(i,\pi) =l)$ to be the value of $X_i(s,\pi)$ under the assumption that there are exactly $l$ w-chains revealed before $M$. 
 
 Now suppose that $L(i,\pi) =l$.
Each of the $l$ w-chains intersects $M$ in a unique element. For any of these $\ell$ w-chains $W$ we know the value of $\mathrm{top}_D(W)$.
The intersection $a$ of $M$ and $W$ is in the downset $D$ if and only if $\mathrm{top}_D(W)\ge a$. Moreover, $a$ is in $D$ if an only if $t$ is at least $a$. Therefore, revealing $W$ determines whether $a$ is in the downset $D$ or not. This means that after revealing $l$ w-chains, there are (at least) $l$ distinct elements on $M$ that are known to be in $D$ or not.

Computing $X_i(s,\pi\mid L(i,\pi) =l)$ directly is complicated, so we estimate it with a simpler random variable $N_l$.


\begin{defi}\label{def:Nl}
Take $n+1$ elements in a cyclic order, such that one of them is $t$.
Pick $l\le n$ distinct elements uniformly at random, indexed in circular order as $a_1, \ldots, a_\ell$  (allowing $a_j=t$ for some $j$).
Let $N_l$ be the length of the circular interval $[a_j,a_{j+1})$ (indexed cyclically) that contains $t$.
\end{defi}

A random variable $X$ is {\it at most $Y$ in the stochastic order} if for all $x$ we have $Pr[X\le x] \leq Pr[Y\le x]$. 

\begin{claim}\label{prop:domination}
For $i \leq n$ and $2\le \ell \le n$, $X_i(s,\pi\mid L(i,\pi) = l)$ is at most $N_l$ in the stochastic order.
\end{claim}

\begin{claim}\label{prop:circle}
For $2\le \ell \le n$,
\[
    E[N_\ell] \le 2 \frac{n+1}{\ell+1}.
\]
\end{claim}

The proof of these claims can be found in Appendix \ref{app:simpletg}.

We are ready to give an upper bound for $E_{\pi}[X_i(s,\pi)]$ using $Pr[L(i,\pi)=l] = \frac{1}{n+1}$ for every $0\le l\le n$ by the uniformity of $\pi$.
Observe that for $i \leq n$
\begin{align}\label{X-bound}
E_{\pi}[X_i(s,\pi)] &
    \le \sum_{\ell=0}^{n} E_{\pi}[X_i(s,\pi) \mid L(i,\pi)=l]\cdot Pr[L(i,\pi)=\ell\,] \nonumber \\
    & \le \sum_{l=0}^1 \frac{n+1}{n+1} + \sum_{\ell=2}^{n} E[N_\ell]\cdot \frac{1}{n+1} \leq 2+ \sum_{\ell=2}^{n} \frac{2}{\ell+1} \le 2\log (n+1) +3.
\end{align}
An analogous argument for $i > n$ considers a chain $W=W_i$ and shows that (\ref{X-bound}) holds when $i>n$ as well.
Applying Corollary~\ref{cor:main} with $E[X_i(s,\pi)]\leq 2\log (n+1) +3$ gives
\begin{align*}
    TG(n) &\leq  \prod_{i=1}^{2n} \max_{s \in S} E_{\pi}[X_i(s,\pi)] \le  (2\log (n+1) +3)^{2n} \leq e^{2n\log\log n+O(n)}.
\end{align*}

This is worse than the bound given in \cite{KGW}, but better than the prior best bound of $2^{n\log n-O(n)}$ from \cite{St}.

\subsection{Improved combinatorial bound for $TG(n)$}\label{sec:tg}

To get a better bound for $|S|$, we apply Lemma~\ref{lem:main} using $N_\ell$ as defined above. 

\begin{claim}\label{prop:PrYisk}
    For $\ell \geq 2$ and any $1\le k\le n$
    \[
    Pr[N_\ell=k]= k\frac{\binom{n-k}{\ell-2}}{\binom{n+1}\ell}.
    \]
\end{claim}

\begin{proof}
    There are  $k$ potential intervals $[a_j,a_{j+1})$  of length $k$ on $C$ that contain $t$.
    Since the two ends $a_j$ and $a_{j+1}$ need to be two of the $\ell$ known elements from $a_1,\dots, a_\ell$, there are $\binom{n+1-(k+1)}{\ell-2}$ options for the remaining $\ell-2$ known elements from $a_1,\dots, a_\ell$.
\end{proof}

\begin{claim}\label{prop:logY}
For $n \geq 1$, $i \leq n$ and any $s\in S$,
    \[E_{\pi}[\log X_i(s,\pi)]\le 1.2038.\]
\end{claim}

The proof of Claim \ref {prop:logY} follows from Claims \ref{prop:domination} and \ref{prop:PrYisk} with a calculation that is similar to (\ref{X-bound}) but is more involved; the details can be found in Appendix \ref{app:logY}.

As in Section \ref{sec:simpletg}, an analogous argument shows that Claim~\ref{prop:logY} holds for $n+1 \leq i \leq 2n$ as well.
Therefore, applying Lemma~\ref{lem:main} to $S \subset M_1 \times \cdots \times M_n \times W_{n+1} \times \cdots \times W_{2n}$  gives
\[
\log |S| \le \sum_{i=1}^{2n} \max_{s\in S} E_{\pi}[\log X_i(s,\pi)] \le 1.2038\cdot 2n = 2.4076n
\]
which implies $TG(n)\le |S| \le e^{2.4076n}\leq 11.11^n$.


\subsection{Asymptotic upper bound for $TG(n)$}\label{sec:asympt}

In the previous sections we have given exact formulas, but in fact, these are unnecessary, as when $n\to \infty$, instead of dealing with $\binom{n+1}l$ options, we can pick each element with probability $x=\frac ln$.
More precisely, let us define the random variable $N_x$ as follows.

\begin{defi}\label{def:Nx}
Let $A$ be a random subset of the integers such that for every $j\in \mathbb Z$ $Pr[j\in A]=x$, independently. 
Let $N_x$ be the length of the shortest interval formed by the elements of $A$ that contains 0, i.e., $N_x=\min_{b\in A, b>0} b - \max_{a\in A, a\le 0} a.$
\end{defi}

Then we have $Pr[N_x=k]=kx^2(1-x)^{k-1}$ because there are $k$ intervals of length $k$ that contain 0, its two endpoints need to be in $A$, and its $k-1$ intermediate points should not be in $A$.
We can now perform an analogous calculations to that of (\ref{eq1}) as follows.
\begin{align}\label{eqnew}
    E_{\pi}[ \log X_i(s,\pi)] &= \sum_{\ell=0}^{n} E_{\pi}[ \log X_i(s,\pi) \mid L(i,\pi) = l] Pr[L(i,\pi) = l] \nonumber \\
    & \le o(1) + \int_0^1 E[\log N_x] \mathop{dx} \nonumber \\
    &= o(1) + \int_0^1 \sum_{k} (\log k) Pr[N_x=k] \mathop{dx} \nonumber \\ 
    &= o(1) + \int_0^1 \sum_{k} (\log k) k x^2(1-x)^{k-1} \mathop{dx} \nonumber \\ 
    &=o(1) +\sum_{k} \frac{2\log k}{(k+1)(k+2)}\le 1.2038+o(1).
\end{align}

Because of the $o(1)$, this gives a slightly weaker bound $TG(n)\le 11.11^n+O(1)$.

\section{Improved bound for $SM(n)$}\label{sec:tg2}

In this section we bound the number of downsets in the rotation poset which, by Theorem~\ref{bijection-thm}, will give a bound on $SM(n)$. Our approach is similar to that in Section~\ref{sec:simple}, but now  instead of using the tangled grid poset, will focus on the rotation poset directly.

\subsection{Properties of the rotation poset}

First we derive some further properties of the rotation poset.

\begin{claim}\label{prop:two}
    The number of m-chains containing
     rotation $\rho \in \mathcal{R}$ is equal to the number of w-chains containing $\rho$, and is at least $2$.
\end{claim}

\begin{proof}
    By Proposition~\ref{kgw-claim}, a rotation $\rho = (u_0v_0, u_1v_1,\dots, u_{k-1}v_{k-1})$ is contained in m-chains $M(u_i)$ and w-chains $W(v_i)$ for $0 \leq i \leq k-1$ where $k\ge 2$.
\end{proof}

 For a rotation $\rho = (u_0v_0, u_1v_1,\dots, u_{k-1}v_{k-1})$, let us
 {\it pair} up those m-chains and w-chains that contain $\rho$, in two ways:\\
 1) $M(u_i)$ and $M(v_i)$ are paired for every $i$, and\\
 2) $M(u_i)$ and $M(v_{i+1})$ are paired for every $i$.

\begin{claim}\label{prop:consecutive}
    Suppose that an m-chain $M(u)$ and a w-chain $W(v)$ are paired for some rotation $\rho$. Then either $M(u)$ and $W(v)$ are paired at exactly one other rotation $\rho'$ which is consecutive to $\rho$ on both $M(u)$ and $W(v)$, or $\rho$ is at the top or at the bottom of both $M(u)$ and $W(v)$.
    Moreover, none of $M(u)$ or $W(v)$ can be paired to another chain at both $\rho$ and $\rho'$.
\end{claim}

\begin{proof}
    By Proposition~\ref{one-pair}, there is at most one rotation where $uv$ appears as an edge. On the other hand, there is at most one rotation $\rho=(u_0v_0, u_1v_1,\dots, u_{k-1}v_{k-1})$
    such that $u=u_i$ and $v=v_{i+1}$. Indeed, after $\rho$ is eliminated, $u$ and $v$ will be matched, so any rotation $\rho'$ immediately above $\rho$ on $M(u)$ or $W(v)$ has $uv$ appearing as an edge.
    The `moreover' part of the statement follows from that each applicant gets a more preferred job after each rotation by definition.
\end{proof}



Unlike in the case of tangled grids, we cannot assume that all the m-chains and w-chains have the same fixed length, but from Claim \ref{prop:consecutive} we can conclude that the length of each chain is at most $n-1$.
Notice, however, that we never really relied on this assumption for tangled grids.
Indeed, the proof of Claim \ref{prop:logY} works for any other length. 
Our proof in this section also works for any length $m-1\le n-1$, in fact, our upper bound would be better for shorter chains.
We will not give a formal proof of this fact, because that would be cumbersome, but it is quite straight-forward to see that an $X_i(s,\pi)$ that belongs to a shorter chain $M$ is always at most as large in the stochastic order as some $X_i'(s,\pi)$ that could belong to some longer chain $M'$; it is enough to just add some further elements to $M$, and at the new elements pair $M'$ with some w-chains that were not paired with $M$ yet.
For simplicity, in the following we work with length $n-1$ instead of arbitrary $m-1\le n-1$.

Let us fix an m-chain $M=M_i$ whose elements are $\rho_1<\cdots<\rho_{n-1}$.
By Claim \ref{prop:consecutive}, for every $1\le j\le n-2$, there is a chain $W$ paired to $M$ both at $\rho_j$ and at $\rho_{j+1}$, this $W$ is not paired to $M$ anywhere else, and the elements $\rho_j$ and $\rho_{j+1}$ are also consecutive on $W$.
This means that if $\textrm{top}_D(W)=\rho_j$, then also $\textrm{top}_D(M)=\rho_j$.
In other words, if $W$ precedes $M$ according to $\pi$, then $X_i(s,\pi)=1$ and thus $\log X_i(s,\pi)=0$.
This will give a significant improvement of our upper bound.

We can also improve our upper bound by considering not only the intersections of $M$ with w-chains, but also with some m-chains.
Here we cannot define a nice pairing, and in fact the intersection structure of the m-chains can be quite complicated.
But from Claim \ref{prop:two} we know that every $\rho_j$ is contained in at least one other m-chain.
Moreover, for every $1\le j\le n-2$ there are two different m-chains, $M'\ne M$ and $M''\ne M$, such that $\rho_j\in M'$ and $\rho_{j+1}\in M''$.
This is because the w-chain $W$, paired to $M$ at $\rho_j$ and at $\rho_{j+1}$, is paired to another m-chain $M'$ at $\rho_j$ and a different m-chain $M''$ at $\rho_{j+1}$ by the `moreover' part of Claim \ref{prop:consecutive}.

Assuming $\mathrm{top}_D(M)=\rho_j$ for some $2\le j\le n-3$, consider 4 m-chains, each different from $M$, that contain $\rho_{j-1}$, $\rho_j$, $\rho_{j+1}$, $\rho_{j+2}$, respectively, and denote them by $M(-1)$, $M(0)$, $M(1)$, $M(2)$.
(It will be easy to see that in the case when $j=1,n-2$ or $n-1$ we get a better bound, but we will not discuss these cases in detail.)
Knowing $\textrm{top}_D(M(h))$ for some $-1\le h\le 2$ will determine whether $\rho_{j+h}\in D$ or not, which can also decrease the number of options left for $\textrm{top}_D(M)$.
We can assume $M(-1)\ne M(0)\ne M(1)\ne M(2)$ because of the above discussion, but it is possible that these 4 m-chains are not all different, e.g., $M(0)=M(2)$ is allowed.

We remark that considering more than 4 m-chains, i.e., increasing the range of the above $h$, would not improve our upper bound further in a straight-forward manner
because, with the natural extension of our above notation, we could have $M(1)=M(3)=M(5)=\ldots$ and $M(2)=M(4)=\ldots$, and in this case $\rho_{j+h}\notin D$ would automatically imply $\rho_{j+h+2}\notin D$, and we would not gain more information from having a larger intersection with the same chain above (or below) $\mathrm{top}_D(M)$.

\subsection{Final calculations}\label{sec:asympt2}

We will follow the analysis of Section \ref{sec:asympt}, as dealing with complicated sums of binomial coefficients, as in Section \ref{sec:tg}, would be very cumbersome, and would not improve the base of the exponent in our upper bound.

Denote the number of w-chains different from $W$ (paired to $M$ at $\mathrm{top}_D(M)$ and at the consecutive element) and preceding $M$ according to $\pi$ by $L(i,\pi)=|\{j\in [n+1,2n] \,:\, \pi^{-1}(j)<\pi^{-1}(i)\}, W_j\ne W|$. 
Each of these $L(i,\pi)$ w-chains are paired to $M$ at exactly two consecutive elements. 
We will give an upper bound for $X_i(s,\pi\mid x=L(i,\pi)n)$, that is, when we assume that there are $xn$ w-chains $W$ preceding $M$, with the help of the random variable $N_x$ defined as follows.

\begin{defi}\label{def:Nx2}
Let $A$ be a random subset of the integers such that for every $j\in \mathbb Z$ $Pr[j\in A]=x$, independently.
Let $B$ be a random subset of $\{-1,0,1,2\}$ such that for every $j\in \{-1,0,1,2\}$ $Pr[j\in B]=x$, independently. 
Let $N_x$ be the length of the shortest interval formed by the elements of $A$ and $B$ that contains 0, i.e., $N_x=\min_{b\in A\cup B, b>0} b - \max_{a\in A\cup B, a\le 0} a$, with probability $1-x$, and let $N_x=1$ with probability $x$.
\end{defi}

Before we proceed, we show that if we similarly define a random variable $N_x'$, with the only difference being that instead of $B$ we use $B'$, and certain events $j\in B'$ need not be independent, then $E[\log N_x']\le E[\log N_x]$.
More precisely, we still require for every $j\in \{-1,0,1\}$ the events $j\in B'$ and $j+1\in B'$ to be independent, but we require for other pairs $j,j'\in \{-1,0,1,2\}$ either that $j\in B'$ is independent from $j'\in B'$, or that $j\in B'$ if and only if $j'\in B'$.
We need this to later argue that $E_\pi[\log X_i(s,\pi)]$ is maximized when the above 4 m-chains are all different.

\begin{claim}\label{prop:Nx'}
 $E[\log N_x']\le E[\log N_x]$ for any $N_x'$ with the above properties.
\end{claim}

The proof of Claim \ref{prop:Nx'} can be found in Appendix \ref{app:Nx'}.

A random variable $X=X(n)$ is \emph{asymptotically at most $Y=Y(n)$ in the stochastic order} if for all $x$ we have $Pr[X\le x]\le Pr[Y\le x]+o(1)$, where $o(1)$ is a constant that tends to 0 as $n\to \infty$.

\begin{claim}\label{prop:domination2}
For $i \leq n$, $L(i,\pi)=\omega(1)$ and $L(i,\pi)n$ integer, $X_i(s,\pi\mid x=L(i,\pi)n)$ is asymptotically at most $N_x$, or at most some $N_x'$, in the stochastic order.
\end{claim}
\begin{proof}
    Let $X_i(s,\pi\mid x=L(i,\pi)n)$ and suppose that $\textrm{top}_D(M)=\rho_j$.
    If the chain $W$ paired to $M$ at $\rho_j$ and $\rho_{j+1}$ precedes $M$, then $|X_i(s,\pi\mid x=L(i,\pi)n)|=1$.
    The probability for this is $x$, since we can generate the permutation $\pi$ by first taking a random permutation of $M$ and the w-chains that are different from $W$, and then insert $W$ (and the rest of the m-chains).
    
    Otherwise, thus with probability $1-x$, $M$ precedes $W$ and we proceed as follows.
    If the w-chain paired to $M$ at both $\rho_{j+h}$ and $\rho_{j+h+1}$  precedes $M$, then we put $h$ to $A$ if $h\ge 1$, and we put $h+1$ to $A$ if $h\le -1$.
    The probability of these events is asymptotically $x$ if $n\to \infty$ by a similar logic as we have used in case of $W$ above.
    
    Finally, fix 4 m-chains, each different from $M$, that contain $\rho_{j-1}$, $\rho_j$, $\rho_{j+1}$, $\rho_{j+2}$, respectively, and denote them by $M(-1)$, $M(0)$, $M(1)$, $M(2)$.
    If for some $-1\le h\le 2$ $M(h)$ precedes $M$, put $h$ to $B$.
    The probability of each of these events is also asymptotically $x$ if $n\to \infty$.
    
    Moreover, if $n\to \infty$, then any bounded collection of the above events is almost independent, meaning that if we replace them by independent events, then the probability that we get a different outcome is $o(1)$. 
    
    By definition, $X_i(s,\pi\mid x=L(i,\pi)n)$ is at most the distance between the two intersections of $M$ with earlier w-chains and the above 4 m-chains; the distance between the largest intersection that is at most $\rho_j$, and the smallest intersection that is larger than $\rho_j$.
    But this gives exactly the definition of $N_x$, or some $N_x'$, depending on which of the above 4 m-chains coincide.
\end{proof}

To repeat the calculation similar to (\ref{eqnew}), we first need to calculate $Pr[N_x=k]$ for $k\ge 2$.
(We do not need $Pr[N_x=1]$ because this will be multiplied by $\log N_x=0$.)
We introduce the notation $I=[\max_{a\in A\cup B, a\le 0} a,\min_{b\in A\cup B, b>0} b)$.
By the definition of $N_x$,
\begin{align*}
Pr[N_x=k]&=(1-x)\sum_{j\le 0} Pr[I=[j,j+h)]\\
&=(1-x)\sum_{j\le 0} Pr[j\in A\cup B]Pr[j+k\in A\cup B]\prod_{h=1}^{k-1} Pr[j+h\notin A\cup B].
\end{align*}

Using the above formula, we can calculate $Pr[N_x=k]$ as shown below.
We treat the cases $k=2$, $k=3$ and $k\ge 4$ separately; we have only provided all the steps in the straight-forward calculations for deriving the respective polynomials for $k=2$.
\begin{align*}
     Pr[N_x=2] &=(1-x)(Pr[I=[-1,1)]+Pr[I=[0,2)])\\
     &=(1-x)(Pr[-1\in A\cup B]Pr[1\in A\cup B]Pr[0\notin A,B]\\
     &~~~+Pr[0\in A\cup B]Pr[2\in A\cup B]Pr[1\notin A,B])\\
     &=(1-x)(2(1-(1-x)^2)(1-(1-x)^2)(1-x)^2)\\
     &=2(1-x)^3(1-(1-x)^2)^2.\\
    ~\\
     Pr[N_x=3] &=(1-x)(Pr[I=[-1,2)]+Pr[I=[-2,1)]+Pr[I=[0,3)])\\
     &=(1-x)^5(1-(1-x)^2)^2+2(1-x)^5(1-(1-x)^2)x.\\
    ~\\
    Pr[N_x=k\ge 4] &=(1-x)(Pr[I=[0,k)]+Pr[I=[-k+1,1)]\\
    &~~~+Pr[I=[-1,k-1)]+Pr[I=[-k+2,2)]
    +\sum_{-k+3\le h\le -2} Pr[I=[h,h+k)])\\
     &=2(1-x)^{k+2}(1-(1-x)^2)x+2(1-x)^{k+3}(1-(1-x)^2)x+(k-4)(1-x)^{k+4}x^2.
\end{align*}

Now we are ready to prove the upper bound for $E_\pi[ \log X_i(s,\pi)]$.

\begin{claim}\label{prop:logY2}
For any $i \leq n$ and $s\in S$, as $n\to \infty$
    \[E_\pi[\log X_i(s,\pi)]\le 0.6331+o(1).\]
\end{claim}
\begin{proof}

\begin{align}\label{eqnew2}
    &E_\pi[ \log X_i(s,\pi)] = \sum_{\ell=0}^{n} E_\pi[ \log X_i(s,\pi) \mid x=L(i,\pi)n] Pr[x=L(i,\pi)n] \nonumber \\
    & \le o(1) + \int_0^1 E[\log N_x] \mathop{dx} \nonumber \\
    &= o(1) + \int_0^1 \sum_{k} (\log k) Pr[N_x=k] \mathop{dx} \nonumber \\ 
    &= o(1) + \int_0^1 (\log 2) 2(1-x)^3(1-(1-x)^2)^2 \nonumber \\ 
    &~~~+(\log 3) ((1-x)^5(1-(1-x)^2)^2+2(1-x)^5(1-(1-x)^2)x) \nonumber \\ 
    &~~~+ \sum_{k\ge 4} (\log k) (2(1-x)^{k+2}(1-(1-x)^2)x+2(1-x)^{k+3}(1-(1-x)^2)x+(k-4)(1-x)^{k+4}x^2) \mathop{dx} \nonumber \\ 
    &=o(1)+(\log 2)\frac{1}{12}+(\log 3)\frac{23}{630}+\sum_{k\ge 4} (\log k)\frac{2k(k+7)+72}{(k+3)(k+5)(k+6)(k+7)}\le 0.6331+o(1). \nonumber\qedhere
\end{align}
\end{proof}

As usual, an analogous argument holds for $n+1 \leq i \leq 2n$ as well.
Therefore, applying Lemma~\ref{lem:main} to $S \subset M_1 \times \cdots \times M_n \times W_{n+1} \times \cdots \times W_{2n}$ gives
\[
\log |S| \le \sum_{i=1}^{2n} \max_{s\in S} E_{\pi}[\log X_i(s,\pi)] \le (0.6331+o(1)) \cdot 2n = 1.2662n+o(n)
\]
which implies $SM(n)\le |S| \le e^{1.2663n}+O(1)\leq 3.55^n+O(1)$.%

\section{Concluding remarks}
Note that in an (untangled) $n \times n$ grid with its elements ordered as a `diamond' (with a unique smallest and largest element), the number of downsets is exactly $\binom{2n}n$.
This makes one think that possibly $TG(n)\le 4^n\cdot \text{poly}(n)$ holds, as was conjectured by the second author.
This, however, turned out to be false; Clay Thomas (personal communication) gave an example showing $TG(n)= \Omega(4.17^n)$ by modifying the construction from \cite{IL} showing $SM(n)= \Omega(2.28^n)$.

It would be nice if in our proofs one could avoid dealing with the m-chains and w-chains separately, and hope that with some tricks the number of components can be reduced this way from $2n$ to $n$.
This would imply that the base of the exponent could be reduced to its square root. Unfortunately, this is too good to be true for $TG(n)$, as the lower bound is larger than the square root of our upper bound, $11.11^n$.
Because of this, most likely for $SM(n)$ there is also no simple way to deal with only one kind of chains.


We do not believe our upper bounds to be tight; the truth is probably closer to the known lower bounds for both $TG(n)$ and $SM(n)$.
In fact, for $SM(n)$, our upper bound can be most likely improved by looking in more detail at the way other chains can intersect a given chain $M$.
With such a tedious analysis, however, it is unlikely that the lower bound can be matched.
Possibly some step by step modification of the rotation poset could show what it needs to be in the extremal case, and that could give a sharp bound.




\subsection*{Acknowledgments}

We would like to thank Padmini Mukkamala for pointing out Claim \ref{prop:circle}, Viktor Harangi and G\'abor Kun for discussions about other proofs of Lemma \ref{lem:main}, \'Agnes Cseh for calling our attention to relevant literature about stable matchings, and Bal\'azs Patk\'os for comments on the first version of this paper. This work has started while the first author visited the second at E\"otv\"os Lor\'and University, supported by grant number LP2017-19/2017 of the Lend\"ulet program of the Hungarian Academy of Sciences (MTA).

\appendix
\section{Proof of Lemma \ref{lem:main}}\label{app:main}

Define the entropy of a random variable $Z$ taking finitely many values as
\[H(Z)= - \sum_z Pr[Z=z]\cdot \log Pr[Z=z].\]

\begin{proof}[Proof of Lemma \ref{lem:main}]
    Fix an arbitrary permutation $\pi$ of $\{1,2,\dots,n\}$ and select
    $s\in S$ uniformly at random. From the basic properties of entropy (see \cite{R}) we have
	\begin{align*}
	  & \log |S| = H(s) 
	= \sum_{i=1}^n H(s_i \mid s_j \textit{ for } j \textit{ satisfying } \pi^{-1}(j)<\pi^{-1}(i)) \\ 
	& \le \sum_{i=1}^n H(s_i \mid X_i(s,\pi)) 
	= \!\sum_{i=1}^n \!\sum_k Pr_s[X_i(s,\pi)=k]\cdot H(s_i \mid X_i(s,\pi)=k) \\
	&\le \sum_{i=1}^n \sum_k Pr_s[X_i(s,\pi)=k] \cdot \log k= \sum_{i=1}^n E_s[\log X_i(s,\pi)].\qedhere
	\end{align*}
\end{proof}	

\section{Proofs of claims from Section \ref{sec:simpletg}}\label{app:simpletg}

\begin{proof}[Proof of Claim \ref{prop:domination}]
Let $a_1 \prec a_2 \prec \cdots \prec a_\ell$ be the intersections of $M$ with the $l$ revealed w-chains. The bottom element of $M$, $a_0$, satisfies $a_0 \prec a_1$ as no w-chain contains $a_0$. If we imagine the elements of $M$ arranged in the natural cyclic order around a cycle $C$, then $a_0 \prec a_1 \prec \cdots \prec a_{\ell}$ partition $C$ into $\ell+1$ intervals of the form
$[a_j,a_{j+1})$ (indexed modulo $\ell+1$).
As $a_0$ is in $D$, there is a maximum $j\geq 0$ such that $a_j$ is in $D$. Thus, 
$t$ must be on the interval $[a_j,a_{j+1})$ of $C$. Let $Y_\ell$ be the length of this interval. Observe that the number of options for $t$, i.e., $X_i(s,\pi\mid L(i,\pi) =l)$, is at most $Y_\ell$.
This allows us to ignore the structure of the poset $P^*$ and estimate $Y_l$.
For clarity, let us give an alternative definition of $Y_l$:

Take $n+1$ elements in a cyclic order, such that one of them is $a_0$, and one of them is $t$ (allowing $a_0=t$).
Pick $l$ distinct elements uniformly at random, indexed in circular order as $a_1, \ldots, a_\ell$ such that none of them is $a_0$ (but allowing $a_j=t$ for some $j$).
Let $Y_l$ be the length of the circular interval $[a_j,a_{j+1})$ (indexed modulo $l+1$).

Note that $Y_l$ depends on the relative position of $a_0$ and $t$.
Now it remains to show that $Y_l$ is at most $N_l$.
This is indeed the case because $N_l$ is essentially the same as $Y_l$, except that $a_0$ is not used when creating the intervals $[a_j,a_{j+1})$, so all these intervals are at least as long in case of $N_l$, as in case of $Y_l$.
The only issue is that it is possible that for some $j$ we have $a_j=a_0$, but in this case we get a distribution that is the same as $Y_{l-1}$, which is at least $Y_l$ in the stochastic order.

Let $Q$ be the event that $a_0 \neq a_j$ for all $1 \leq j \leq l$. Then
\begin{align*}
Pr[N_l\le x]&=Pr[N_l\le x\mid Q]Pr[Q]+Pr[N_l\le x \mid \overline{Q} ](1-Pr[Q])\\
&\ge Pr[Y_l\le x]Pr[Q] + Pr[Y_{l-1}\le x](1-Pr[Q])\\
&\ge Pr[Y_l\le x]Pr[Q] + Pr[Y_{l}\le x](1-Pr[Q])=Pr[Y_l\le x].\qedhere
\end{align*}

\end{proof}

\begin{proof}[Proof of Claim \ref{prop:circle}]
We will bound $E[N_\ell \mid t\ne a_j \, \forall 1\le j\le l]$ and $E[N_\ell \mid \exists j: t=a_j]$ separately.

First, select $\ell+1$ elements $a'_0 \prec a'_1 \prec \cdots \prec a'_{\ell}$ uniformly at random from all $n+1$ elements of $C$. This partitions $C$ into $\ell+1$ intervals of the form $[a'_j,a'_{j+1})$.
The expected length of each of these intervals is exactly $\frac{n+1}{\ell+1}$.
Now, pick one of the $a_{j_0}'$ uniformly at random, declare $t=a_{j_0}'$ and reindex the rest of the $a_j'$ into $a_1,\ldots,a_l$.
This generates the same distribution of $t,a_1,\ldots,a_l$ as in $N_l$ given $t\ne a_j \, \forall 1\le j\le l$, so the length of the interval $[a_j,a_{j+1})$ containing $t$ will be $E[N_\ell \mid t\ne a_j \, \forall 1\le j\le l]$.
On the other hand, it is also the union of two intervals, $[a_{j_0-1}',a_{j_0}')$ and $[a_{j_0}',a_{j_0+1}')$, thus its expected length is $2\frac{n+1}{\ell+1}$.
This proves $E[N_\ell \mid t\ne a_j \, \forall 1\le j\le l]=2\frac{n+1}{\ell+1}$.

The other case can be computed as follows.
$E[N_\ell \mid \exists j_0: t=a_{j_0}]$ is the same as the expected value of the smallest of $l-1$ different numbers picked uniformly at random from $[1,n]$.
These divide the $n+1$ elements into $l$ equal intervals, or we can compute the expectation as follows.
\[
E[N_\ell \mid \exists j_0: t=a_{j_0}]=\sum_{k=1}^{n-l+2} Pr[\forall j\ne j_0: a_j\ge k]
=\sum_{k=1}^{n-l}\frac{\binom{n-k+1}{l-1}}{\binom{n}{l-1}}
=\frac{\binom{n+1}{l}}{\binom{n}{l-1}}=\frac{n+1}{\ell}.
\]

If $l\ge 1$, then $\frac1l\le \frac2{l+1}$, so both
$E[N_\ell \mid t\ne a_j \, \forall 1\le j\le l]$ and $E[N_\ell \mid \exists j: t=a_j]$ are less than $2\frac{n+1}{\ell+1}$, thus so is $E[N_\ell]$.
\end{proof}

\section{Proof of Claim \ref{prop:logY}}\label{app:logY}

\begin{proof}[Proof of Claim \ref{prop:logY}]
\begin{align}\label{eq1}
    E_{\pi}[ \log X_i(s,\pi)] &= \sum_{\ell=0}^{n} E_{\pi}[ \log X_i(s,\pi) \mid L(i,\pi) = l] Pr[L(i,\pi) = l] \nonumber \\
    & \le \sum_{l=0}^1 \frac{\log(n+1)}{n+1} + \sum_{\ell=2}^{n} E[\log N_\ell] \frac{1}{n+1} \nonumber \\
    &\le \frac {2\log (n+1)}{n+1}  +\frac 1{n+1} \sum_{\ell=2}^{n} \sum_k k \frac{\binom{n-k}{\ell-2}}{\binom {n+1}\ell}\log k \nonumber \\ 
    &=\frac {2\log (n+1)}{n+1} +\frac 1{n+1}\sum_{k=2}^n k\log k \sum_{\ell=2}^{n} \frac{\binom{n-k}{\ell-2}}{\binom {n+1}\ell}.
\end{align}

To bound the inner sum above we use the following identity.

\begin{lem}[Whitworth Identity \cite{W}]\label{lem:whitworth}
If $m\ge 0$ and $n\ge m+a$, then
\[
\sum_{j=0}^m \frac{\binom{m}{j}}{\binom{n}{j+a}} = \frac{n+1}{(a+1)\binom{n-m+1}{a+1}}.
\]
\end{lem}

As the above identity appeared as an exercise in a book published over 100 years ago \cite{W}, we provide a proof below. 

Using the Whitworth Identity with $a=2$ we get
\begin{align*}
 &\sum_{k=2}^n k\log k \sum_{\ell=2}^{n} \frac{\binom{n-k}{\ell-2}}{\binom {n+1}\ell}
 = \sum_{k=2}^n k \log k \sum_{j=0}^{n-k} \frac{\binom{n-k}{j}}{\binom{n+1}{j+2}}\\
 &= \sum_{k=2}^n k \log k \left(\frac{n+2}{3 \binom{k+2}{3}} \right) = 2(n+2) \sum_{k=2}^n \frac{\log k}{(k+1)(k+2)}. 
\end{align*}

Putting this back to Inequality (\ref{eq1}), we obtain 
\begin{align*}
    E_{\pi}[\log X_i(s,\pi)]&\le \frac 2{n+1} \log (n+1) + 2\frac {n+2}{n+1} \sum_{k=2}^n \frac{\log k}{(k+1)(k+2)}.
\end{align*}

From here it is easy to confirm numerically that $E_{\pi}[ \log X_i(s,\pi)] \leq 1.2038$ for every $n\geq 1$.
\end{proof}




\begin{proof}[Proof of the Whitworth Identity]
    First count the number of ways to choose $m$ objects with repetition from the union of sets of sizes $a+1$ and $b+1$, respectively, to establish
    \begin{equation}\label{VDM-id}
       \binom{a+b+m+1}{m} = \sum_{j=0}^{m} \binom{a+j}{j} \binom{b+m-j}{m-j}.
    \end{equation}

    Now let us count the number of strings of length $a+b+m$ using exactly $a$ zeros, $m-j$ ones, $j$ twos and $b$ threes in two ways:
    \begin{equation}\label{three-count}
        \binom{a+b+m}{a+m}\binom{a+m}{m} \binom{m}{j} = \binom{a+b+m}{a+j}\binom{a+j}{j} \binom{b+m-j}{m-j}.    
    \end{equation}
    
    Rearranging terms in (\ref{three-count}) and substituting into (\ref{VDM-id}) gives
    \begin{align*}      
    \binom{a+b+m+1}{m} &= \sum_{j=0}^{m} \frac{\binom{a+b+m}{a+m}\binom{a+m}{m}\binom{m}{j}}{\binom{a+b+m}{a+j}}\\ &=
    \binom{a+b+m}{a+m}\binom{a+m}{m}\sum_{j=0}^{m} \frac{\binom{m}{j}}{\binom{a+b+m}{a+j}}.
    \end{align*}
    Now making the substitution $n=a+b+m$ and rearranging terms gives
    $$
    \sum_{j=0}^m \frac{\binom{m}{j}}{\binom{n}{j+a}} = \frac{\binom{n+1}{m}}{\binom{n}{a+m}\binom{a+m}{m}} = \frac{n+1}{(a+1)\binom{n-m+1}{a+1}}
    $$
    where the last equality comes from simplification.
\end{proof}

\section{Proof of Claim \ref{prop:Nx'}}\label{app:Nx'}

\begin{proof}[Proof of Claim \ref{prop:Nx'}]
For the proof, we need the following simple inequality.

\begin{prop}\label{prop:4chan}
    Let $a_0,a_1,a_2>0$.
    Suppose that for two independent random variables, $X_1$ and $X_2$, we have $Pr[X_1=0]=Pr[X_2=0]=x$ and $Pr[X_1=a_1]=Pr[X_2=a_2]=(1-x)$.
    Suppose that the random variables $X_1'$ and $X_2'$ have the same distribution as $X_1$ and $X_2$, respectively, but they completely determine each other, i.e., $X_1'=0$ if and only if $X_2'=0$.
    Then \[E[\log(a_0+X_1+X_2)]\ge E[\log(a_0+X_1'+X_2')].\]
\end{prop}
\begin{proof}
This follows from the following application of Jensen's inequality for the $\log$ function.
\begin{align*}
    E[\log(a_0+X_1+X_2)]&=x^2\log a_0 + x(1-x)(\log(a_0+a_1)+\log(a_0+a_2)) + (1-x)^2\log(a_0+a_1+a_2)\\
    &\ge
    x^2\log a_0 + x(1-x)(\log a_0 +\log(a_0+a_1+a_2)) + (1-x)^2\log(a_0+a_1+a_2)\\
    &=x\log a_0 + (1-x)\log(a_0+a_1+a_2)
    =E[\log(a_0+X_1'+X_2')].\qedhere
\end{align*}
\end{proof}

Now we are ready to show $E[\log N_x']\le E[\log N_x]$. 
Suppose that, say, $0\in B'$ if and only if $2\in B'$, but $-1\in B'$ and $1\in B'$ are independent from all other events.
(The other cases go similarly.)
First, randomly determine for each $j\in\mathbb Z$ whether $j\in A$ or not, and also whether $-1,1\in B'$ or not, and let these sets be $A_0$ and $B_0$.
Notice that Proposition \ref{prop:4chan} with $a_0=2$, $a_1=\max_{a\in A_0\cup B_0, a\le 0}$, $a_2=\min_{b\in A_0\cup B_0, b>0} b$ gives
\[E[\log N_x'\mid A=A_0,B'\setminus\{-1,1\}=B_0]\le E[\log N_x\mid A=A_0,B\setminus\{-1,1\}=B_0].\]
By averaging over all possible $A_0,B_0$ with the appropriate weights, we get $E[\log N_x']\le E[\log N_x]$, finishing the proof of Claim \ref{prop:Nx'}.
\end{proof}

\end{document}